\documentclass[11pt,a4paper]{amsart}
\usepackage{hyperref}
\usepackage{graphicx} 
\usepackage{wasysym}
\usepackage[mathscr]{eucal}

\usepackage[draft]{changes}

\usepackage{amsmath}

\usepackage{amssymb}
\usepackage{amsthm}
\usepackage{mathtools} 

\usepackage{amsfonts}
\usepackage{tikz}
\usepackage{enumerate}
\usepackage{csquotes}

\newtheorem{theorem}{Theorem}[section]
\newtheorem{lemma}[theorem]{Lemma}
\newtheorem{proposition}[theorem]{Proposition}
\newtheorem{corollary}[theorem]{Corollary}

\theoremstyle{definition}
\newtheorem{definition}[theorem]{Definition}
\newtheorem{example}[theorem]{Example}

\newtheorem{remark}[theorem]{Remark}

\providecommand{\ceil}[1]{{}^\lceil #1 {}^\rceil}

\begin{document}
	
	\title[Aggregation functions as lax morphisms of quantales]{Aggregation functions as lax morphisms of quantales}
	
	\author[A. Fructuoso-Bonet]{Alejandro Fructuoso-Bonet}
	\address[A. Fructuoso-Bonet]{Instituto Universitario de Matem\'atica Pura y Aplicada, Universitat Polit\`ecnica de Val\`encia, Camino de Vera s/n, 46022 Valencia, Spain}
	\email{afrubon@posgrado.upv.es}
	\author[J. Rodr\'{\i}guez-L\'opez]{Jes\'us Rodr\'{\i}guez-L\'opez}
	\address[J. Rodr\'{\i}guez-L\'opez]{Instituto Universitario de Matem\'atica Pura y Aplicada, Universitat Polit\`ecnica de Val\`encia, Camino de Vera s/n, 46022 Valencia, Spain}
	\thanks{The last author's research
		is part of the project PID2022-139248NB-I00 funded by MICIU/AEI/10.13039/501100011033 and ERDF/EU}
	\email{jrlopez@mat.upv.es}
	
	

	\begin{abstract}
		We will generalize the concept of aggregation function for mathematical structures as a certain function between quantales. In fact, these functions turn to be exactly the lax morphism of quantales. This provides a global framework for the study of aggregation functions. As a consequence of our theory, we are able to deduce several known results about the aggregation of metrics and fuzzy metrics.
	\end{abstract}
	
	\date{This is the author's version of the article accepted for publication in \textit{Fuzzy Sets and Systems}. 
		The final published version is available at Elsevier via DOI: \href{https://doi.org/10.1016/j.fss.2025.109395}{10.1016/j.fss.2025.109395}.}
	\maketitle

	%
	\section{Introduction}
	%
	
	In a general sense, aggregation means a process of fusing a set of numbers into a single one. This simple concept encompasses binary operations like addition or multiplication and other ones more sophisticated like aggregation employing integrals with respect to a nonadditive measure (see, for example, \cite{BookGraMariMesiPap}). This wide spectrum of aggregation methods opened a broad variety of fields of applications like decision theory, operations research, computer science, etc. Functions making an aggregation are usually known as \emph{aggregation functions}, and nowadays, there exists a well-established theory about them \cite{BookBeliPraCalvo,BookGraMariMesiPap,BookTorraNaru}. Moreover, due to their applications, it is a very active area of research.
	
	In addition to combining numbers, many researchers have studied how to use aggregation functions to fuse other mathematical structures. Maybe, the first example of this method is due to Dobo\v{s} and his collaborators \cite{BorsikDobos81,Dobos95,Dobos98} who characterized those functions $F:[0,+\infty)^I\to [0,+\infty)$ satisfying that given a family of metric spaces $\Big\{(X_i,d_i):i\in I\Big\}$ then $F\circ d$ is a metric on $\prod_{i\in I}X_i$ where $d:\prod_{i\in I}X_i\times \prod_{i\in I} X_i\to [0,+\infty)^I$ is given by
	$$d(x,y)=(d_i(x_i,y_i))_{i\in I}$$
	for all $x=(x_i)_{i\in I}, y=(y_i)_{i\in I}\in\prod_{i\in I} X_i.$
	
	Later on, Pradera and Trillas \cite{PraTri02} (see also \cite{PraTriCasti02a,MayorVale19}) studied a related problem characterizing the functions $F:[0,+\infty)^I\to [0,+\infty)$ verifying that for a family $\{d_i:i\in I\}$ of pseudometrics on a set $X$ then $F\circ \delta$ is a metric on $X$ where $\delta:X\times X\to [0,+\infty)^I$ is given by
	$$\delta(x,y)=(d_i(x,y))_{i\in I}$$
	for all $i\in I.$
	
	Afterward, other researchers analyzed the problem of the aggregation of other mathematical structures like quasi-metrics \cite{MassaVale13,MayorVale10,MinyaVale19}, norms \cite{HerMos91,MartinMayorVale11,PRL21}, fuzzy binary relations \cite{DrewDud06,DrewDud07,MayorReca16,SamiMesiBoden02}, fuzzy quasi-pseudometrics \cite{PRLVale21}, probabilistic quasi-uniformities \cite{PRL20b}, etc.

Current research on aggregation functions has developed on a case-by-case basis, leaving open the question of whether a unified theoretical framework could encompass these diverse aggregation processes. A review of previous investigations indicates that certain properties frequently appear in their characterizations, suggesting the potential existence of such a general theory. To address this gap, our research seeks to discover this framework. By leveraging this abstraction, we aim to provide a systematic and unified approach to aggregation functions that not only recovers and extends previous findings but also offers new insights into their properties. To achieve this, we first need a concept that encompasses the mathematical structures whose aggregation has already been studied.

	In the celebrated paper \cite{Lawv73}, Lawvere showed that the theory of enriched categories could be used to unify the theory of metric spaces and ordered spaces. This original idea has given rise to numerous investigations, especially in Quantitative Domain Theory, to develop computer domains. Usually, these investigations do not have the level of generality of Lawvere's work since they don't  use categories enriched in a complete symmetric monoidal closed category but enriched in a quantale. This reduces the complexity of the theory but is still general enough to comprehend the most compelling examples, such as metric and partially ordered spaces.
	
	In this regard, we must mention the work of Flagg and Kopperman \cite{Flagg92,Flagg97,FlaggKopp97,Kopp88}. They looked for a generalized concept of metric space that allows to obtain that every topological space is a generalized metric space (see also the starting work by Trillas and Alsina \cite{TrillasAlsina78}).  In this way, they used enriched categories over a quantale, giving rise to the concept of continuity space that includes, as particular cases, partially ordered sets, quasi-pseudometric spaces, probabilistic metric spaces, etc.

	As a result, it appears that enriched categories over a quantale could may serve as an effective framework for developing a generalized theory of aggregation functions of mathematical structures. In this way, the primary objective of our study is to confirm this by interpreting classical aggregation theory within the context of $\mathscr{V}$-categories, that is, categories enriched over a quantale $\mathscr{V}$. We believe this approach holds significant potential, as it can be applied to any mathematical structure that can be viewed as a $\mathscr{V}$-category. Consequently, the theory may also find applications in various fields where aggregation functions play an important role, such as decision-making, multi-criteria analysis, economics, and computer science, among others. Our work establishes a theoretical foundation that could facilitate future advancements in these interdisciplinary domains, bridging the gap between categorical methods and classical aggregation theory.

	This article is organized as follows. We devote Section \ref{sec:vcat} to recall the basic concepts about quantales and $\mathscr{V}$-categories emphasizing the role of certain prominent quantales for obtaining classical mathematical structures as $\mathscr{V}$-categories. In Section \ref{sec:pf}, we introduce the concept of preserving function (Definition \ref{def:preserving}) between quantales that will be shown to be an appropriate generalization of the concept of aggregation function. For example, we will demonstrate in Lemma \ref{lem:preserving_eqpmaf} that preserving functions between Lawvere quantales are extended quasi-pseudometric aggregation functions. One of the paper's main results is the characterization of preserving functions as lax morphisms of quantales (Theorem \ref{thm:plmq}). As a consequence of this result and others, we obtain some theorems already proved in previous researchers. In particular, we obtain results about aggregation functions of quasi-pseudometrics and fuzzy quasi-pseudometrics. We also fill in the existent gap in the literature about the characterization of functions aggregating quasi-pseudometrics. Furthermore, this more general point of view allows understanding the difference between the characterizations of quasi-metric aggregation functions and metric aggregation functions (see Remark \ref{rem:diff}).

	%
	\section{Categories enriched over quantales}\label{sec:vcat}
	%
	
	In this section we review the basic theory of quantales and categories enriched over quantales. Our basic references are \cite{BookMonoidalTopology,BookQuantales}.

	\begin{definition}[\mbox{\cite[Section II.1.10]{BookMonoidalTopology},\cite[Section 2.3]{BookQuantales}}]
		A \textbf{quantale} $(\mathscr{V},\preceq,\ast)$ is a complete lattice $(\mathscr{V},\preceq)$ such that $\ast:\mathscr{V}\times \mathscr{V}\to \mathscr{V}$ is an associative binary operation which distributes over suprema:
		\begin{align*}
			u\ast \left(\bigvee_{i\in I} v_i\right)&=\bigvee_{i\in I}(u\ast v_i),\\
			\left(\bigvee_{i\in I} v_i\right)\ast u&=\bigvee_{i\in I}(v_i\ast u).
		\end{align*}
		If $\ast$ is also commutative then $(\mathscr{V},\preceq,\ast)$ is a \textbf{commutative quantale}.
		
		A quantale is called \textbf{unital} if $\ast$ has a unit $1_\mathscr{V}.$ A unital quantale is \textbf{integral} if the unit is the top element of $\mathscr{V}.$
		
	\end{definition}

	\begin{remark}
		A unital quantale can be viewed as a closed monoidal ca\-tegory, which is also symmetric when the quantale is commutative (see, for example, \cite{BookMonoidalTopology}).
	\end{remark}

	In the rest of the paper, we will only consider commutative integral quantales, although for simplicity we use only the terminology quantale.
	
	\begin{remark}
		Notice that in an integral quantale $(\mathscr{V},\preceq,\ast)$ we have that $u\ast v\preceq u\wedge v$ for all $u,v\in \mathscr{V}.$ In fact, $u=u\ast (v\vee 1_\mathscr{V})=(u\ast v)\vee (u\ast 1_\mathscr{V})=(u\ast v)\vee u$ so $u\ast v\preceq u.$ In a similar way, $u\ast v\preceq v.$
	\end{remark}

	\begin{example}[\mbox{\cite[Example II.1.10.1]{BookMonoidalTopology}}]\label{ex:quantales}\
		\begin{enumerate}[(a)]
			\item Let $\ast$ be a triangular norm (t-norm for short) on $[0,1]$ (see, for example, \cite{BookTN}), that is, an associative, commutative binary operation $\ast\colon[0,1]\times [0,1]\rightarrow [0,1]$ with unit $1$, such that $a\ast b\leq c\ast d$ whenever $a\leq c$ and $b\leq d$, with $a,b,c,d\in [0,1]$. If $\ast$ is left-continuous then $([0,1],\leq,\ast)$ is a commutative integral quantale, where $\leq$ is the usual order.
			\item Let $\mathbf{2}=\{0,1\}$ be a set with two different elements endowed with the usual order. If $\ast$ is an arbitrary t-norm then $(\mathbf{2},\leq,\ast)$ is a commutative integral quantale.
			\item  Let us consider the opposite order $\leq^\text{\normalfont op}$ on the extended real line $[0,+\infty]$, that is, $x\leq^\text{\normalfont op} y$ if and only if $y\leq x$. If the sum $+$ on the real numbers is extended to $+\infty$ as usual, then $\mathsf{P}_+=([0,+\infty],\leq^\text{\normalfont op},+)$ is a commutative integral quantale called the \textbf{Lawvere quantale} \cite{CookWeiss22} (see also \cite[Example II.1.10.1.(3)]{BookMonoidalTopology}).


			\item Let $\Delta_+$ be the family of distance distribution functions given by
			$$\Delta_+=\{f:[0,\infty]\to [0,1]: f\text{ is isotone and left-continuous} \}$$
			where left-continuous means that $f(x)=\bigvee_{y<x}f(y)$ for all $x\in [0,\infty]$ (as usual, $\bigvee\varnothing =0$). Then $\Delta_+$ is a complete lattice when it is endowed with the pointwise order $\leq$. 
			Moreover, given a left-continuous t-norm $\ast$, $f,g\in\Delta_+ $ and $t\in [0,\infty]$ define $f\circledast g\in \Delta_+ $ as
			$$(f\circledast g)(t)=\bigvee_{r+s\leq t}f(r)\ast g(s)=\bigvee_{r+s= t}f(r)\ast g(t).$$
			Then $(\Delta_+,\leq,\circledast)$ is a quantale (see \cite{Flagg97,HofReis13,JagerShi19}) where the unit is $f_{0,1}:[0,\infty]\to [0,1]$ given by
			$$f_{0,1}(t)=\begin{cases}
				0&\text{ if }t=0,\\
				1&\text{ otherwise},
			\end{cases}$$
			for all $t\in [0,\infty].$ For simplicity, we will denote the quantale $(\Delta_+,\leq,\circledast)$ by $\Delta_+(\ast).$
		\end{enumerate}	
	\end{example}

	\begin{remark}\label{rem:product_quantales}
		Let $\Big\{(\mathscr{V}_i,\preceq_i,\ast_i):i\in I\Big\}$ be an arbitrary family of quantales. Let $\preceq:=\Pi_{i\in I}\preceq_i$ be the product partial order componentwisely defined and $\ast:=\Pi_{i\in I}~\ast_i :(\prod_{i\in I}\mathscr{V}_i)\times(\prod_{i\in I}\mathscr{V}_i)\to \prod_{i\in I}\mathscr{V}_i$ be the componentwise operation given by $(u\ast v)_i=u_i\ast_i v_i$ for all $u=(u_i)_i, v=(v_i)_i\in \prod_{i\in I}\mathscr{V}_i$ and all $i\in I.$ Then $(\prod_{i\in I}\mathscr{V}_i,\preceq,\ast)$ is also a quantale. If all the quantales of the family $\Big\{(\mathscr{V}_i,\preceq_i,\ast_i):i\in I\Big\}$ are equal to $(\mathscr{V},\preceq,\ast)$ we will abuse of notation by denoting also by $\preceq$ and $\ast$ the partial order and the operation on $\mathscr{V}^I.$ In particular, we will denote by $\mathsf{P}_+^I$ the quantale $(\prod_{i\in I} [0,+\infty],\leq^\text{\normalfont op},+)$ and by $\Delta^I_+(\ast)$ the quantale $(\prod_{i\in I} \Delta_+,\leq,\circledast).$
	\end{remark}

	We next recall the central concept of the paper: a $\mathscr{V}$-category, a category enriched over a quantale. As we will see in the next section, certain functions between $\mathscr{V}$-categories can been considered as suitable generalizations of aggregation functions.
	
	\begin{definition}[[\mbox{\cite[Section III.1.3]{BookMonoidalTopology}, c.f. \cite[Definition 3.1]{FlaggKopp97}}]
		Let $(\mathscr{V},\preceq,\ast)$ be a quantale. A \emph{$\mathscr{V}$-category} is a pair $(X,a)$ where $X$ is a nonempty set and $a:X\times X\to\mathscr{V}$ is a map such that
		\begin{itemize}
			\item[(VC1)] $1_\mathscr{V}\preceq a(x,x)$ for all $x\in X;$
			\item[(VC2)] $a(x,z)\ast a(z,y)\preceq a(x,y)$ for all $x,y,z\in X.$
		\end{itemize}
		Moreover, a $\mathscr{V}$-category $(X,a)$ is said to be:
		\begin{itemize}
			\item \emph{separated} if given $x,y \in X$, whenever $1_\mathscr{V}\preceq a(x,y)$ and $1_\mathscr{V}\preceq a(y,x)$ then $x=y.$
			\item \emph{symmetric} if $a(x,y)=a(y,x)$ for all $x,y\in X.$
		\end{itemize}
		
		A function $f:(X,a)\to (Y,b)$ between two $\mathscr{V}$-categories is called a \emph{$\mathscr{V}$-functor} if
		$$a(x,y)\preceq b(f(x),f(y))$$
		for all $x,y\in X.$
		
		The category whose objects are $\mathscr{V}$-categories and whose morphisms are $\mathscr{V}$-functors will be denoted by $\mathscr{V}$-$\mathsf{Cat}.$
		
	\end{definition}
	
	\begin{remark}
		We observe that every topological space can be viewed as a $\mathscr{V}$-category (see \cite{Flagg97,Kopp88}).
	\end{remark}
	
	We next collect some fundamental examples of $\mathscr{V}$-categories.
	
	\begin{example}[\mbox{\cite[Examples III.1.3.1]{BookMonoidalTopology},\cite{HofReis13}}]\label{ex:vcat}\
		\begin{enumerate}[(1)]
			\item $\mathbf{2}$-categories are preordered sets meanwhile $\mathbf{2}$-functors are isotone maps. In fact, given a $\mathbf{2}$-category $(X,a)$, the binary relation $\preceq_a$ given by $x\preceq_a y$ if and only if $a(x,y)=1$ is a preorder on $X.$ 
			
			In a converse way, if $(X,\preceq)$ is a partially ordered set then $a_\preceq:X\times X\to\mathscr{V}$ given by $a_{\preceq}(x,y)=1$ if $x\preceq y$ and $a_{\preceq}(x,y)=0$ otherwise, turns $(X,\preceq)$ into a $\mathbf{2}$-category $(X,a_\preceq)$.
			Therefore, $\mathbf{2}$-$\mathsf{Cat}$ is isomorphic to the category $\mathsf{POrd}$ of preordered sets. 
			
			Obviously, the category of separated $\mathbf{2}$-categories is isomorphic to the category of partially ordered sets.

			\item  $\mathsf{P}_+$-categories are equivalent to extended quasi-pseudometric spaces and $\mathsf{P}_+$-morphisms are equivalent to non-expansive maps, where $\mathsf{P}_+$ is the Lawvere quantale (see Example \ref{ex:quantales}).
			
			\noindent In fact, axioms of a $\mathsf{P}_+$-category are written in this case as
			\begin{quote}
				\begin{itemize}
					\item[(VC1)] $ 0\leq^{op} a(x,x)\Leftrightarrow  0\geq a(x,x)\Leftrightarrow 0=a(x,x);$
					\item[(VC2)] $a(x,y)+a(y,z)\leq^{op}a(x,z)\Leftrightarrow a(x,z)\leq a(x,y)+a(y,z),$
				\end{itemize}
			\end{quote}
			so $a$ is an extended quasi-pseudometric on $X.$
			
			It is also clear that an extended quasi-pseudometric is a $\mathsf{P}_+$-category, so $\mathsf{P}_+$-categories and extended quasi-pseudometric spaces are equivalent concepts. 
			
			Similarly, separated $\mathsf{P}_+$-categories and symmetric separated $\mathsf{P}_+$-categories are equivalent to extended quasi-metric spaces and extended metric spaces respectively.
			
			Moreover, a $\mathsf{P}_+$-functor $f:(X,a)\to (Y,b)$ verifies
			$$a(x,y)\geq b(f(x),f(y))$$
			for all $x,y\in X$, that is, $f$ is a non-expansive mapping. Hence, $\mathsf{P}_+$-$\mathsf{Cat}$ is isomorphic to the category $\mathsf{QPMet}$ of quasi-pseudometric spaces. 
			
			\item If we consider the quantale $\Delta_+(\ast):=(\Delta_+,\leq,\circledast
			)$ (see Example \ref{ex:quantales}), where $\ast$ is a continuous t-norm, then the category $\Delta_+(\ast)$-$\mathsf{Cat}$ is isomorphic to the category $\mathsf{FQPMet}(\ast)$ of fuzzy quasi-pseudometric spaces with respect to $\ast$ and fuzzy nonexpansive maps. 
			
			Notice that if $(X,a)$ is a $\Delta_+(\ast)$-category then $M_a:X\times X\times [0,+\infty]
			\to [0,1]$ defined as $M_a(x,y,t)=a(x,y)(t)$ for every $x,y\in X,$ $t\in [0,+\infty]$, that is, the evaluation of $a(x,y)$ at $t$, satisfies:
			\begin{quote}
				\begin{itemize}
					\item[(FM1)] $M_a(x,y,0)=a(x,y)(0)=0$ for all $x,y\in X;$
					\item[(FM2)] $M_a(x,x,t)=1$ for all $t>0;$
					\item[(FM3)] $M_a(x,z,s)\ast M_a(z,y,t)\leq M(x,y,s+t)$ for every $x,y,z\in X$, $s,t\in [0,+\infty];$
					\item[(FM4)] $M_a(x,y,\cdot)=a(x,y)$ is left-continuous for every $x,y\in X.$
				\end{itemize}
			\end{quote}
			The previous axioms are exactly the conditions that a pair $(M,\ast)$ must satisfy for being a fuzzy quasi-pseudometric on $X$ (see, for example, \cite{GregRo04}), where $M:X\times X\times [0,+\infty)\to [0,1]$ and $\ast$ is a continuos t-norm. Observe that $M_a(x,y,t)$ also exists when $t=\infty$ and $M_a(x,y,\infty)=\bigvee_{0\leq t<\infty} M_a(x,y,t).$ This is not required in the classical definition of a fuzzy quasi-pseudometric but there is no loss of generality if we add this property.\\
			Conversely, let $(M,\ast)$ be a fuzzy quasi-pseudometric on $X.$ If we define $m:X\times X\to\Delta_+(\ast)$ as
			$$m(x,y)(t)=\begin{cases}
				M(x,y,t) &\text{ if } 0\leq t < +\infty,\\
				\bigvee\limits_{0\leq s<+\infty} M(x,y,s)&\text{ if } t=+\infty,
			\end{cases}$$
			for every $x,y\in X$, $t\in [0,+\infty]$, then it is straightfoward to check that $(X,m)$ is a $\Delta_+(\ast)$-category. Observe that $m$ is the exponential mate $\ceil{M}$ of the extension of $M$ to $X\times X\times [0,+\infty].$
			
		\end{enumerate}
	\end{example}
	
	\begin{example}\label{ex:prodqpd}\
		\begin{enumerate}[(1)]
			\item Let $\Big\{(X_i,d_i):i\in I\Big\}$ be a family of (extended) (quasi-)(pseudo)metric spaces which can be considered as $\mathsf{P}_+$-categories. Then $(\prod_{i\in I} X_i,d_\Pi)$ is a $\mathsf{P}_+^I$-category where $d_\Pi:=\Pi_{i\in I} d_i:\prod_{i\in I} X_i\times \prod_{i\in I} X_i\to [0,+\infty]^I$ is the Cartesian product of the mappings $\{d_i\}_{i\in I}$ given by $$d_\Pi((x_i)_i,(y_i)_i)=(d_i(x_i,y_i))_i$$ for every $(x_i)_i,(y_i)_i\in \prod_{i\in I} X_i.$
			\item Let $X$ be a nonempty set and $\{d_i:i\in I\}$ be a family of (extended) (quasi-pseudo)metrics on $X.$ Then $(X,d_\triangle)$ is a $\mathsf{P}_+^I$-category where $d_\triangle:=\triangle_{i\in I} d_i:X\times X\to [0,+\infty]^I$ is the diagonal of the mappings $\{d_i\}_{i\in I}$ given by $$d_\triangle(x,y)=(d_i(x,y))_i$$ for every $x,y,\in X.$
			
			\item  If $(X,a)$ is a $\mathsf{P}_+^I$-category, for each $i\in I$ let $a_i:X\times X\to [0,+\infty]$ be the $i$th-coordinate function of $a$, that is, $a_i(x,y)=(a(x,y))_i$ for every $x,y\in X.$ Then $\Big\{(X,a_i):i\in I\Big\}$ is a family of $\mathsf{P}_+$-categories, that is, a family of extended quasi-pseudometric spaces.
			
		\end{enumerate}
	\end{example}
	
	\begin{remark}\label{rem:prodfqpd}
		If in the above example, we consider the quantale $\Delta_+(\ast)$ instead of $\mathsf{P}_+$, we obtain analogous conclusions for fuzzy quasi-pseudometrics. We will use a similar notation in this case.
	\end{remark}

	A lax homomorphism of monoidal categories \cite{BookMonoidalTopology} (see also \cite[p. 164]{McLane98}), referred to as closed functors in \cite[p. 149]{Lawv73},  simplifies to the following when the monoidal categories are quantales:
	
	\begin{definition}[\mbox{\cite[Section II.1.10]{BookMonoidalTopology},\cite{HofReis13}}]
		A map $F:(\mathscr{V},\preceq,\ast)\to(\mathscr{W},\curlyeqprec,\star)$ between two quantales is said to be a \emph{lax morphism of quantales} if
		\begin{itemize}
			\item $u\preceq v$ implies $F(u)\curlyeqprec F(v)$ for all $u,v\in\mathscr{V};$\hfill (isotone)
			\item $1_\mathscr{W}\curlyeqprec F(1_\mathscr{V});$
			\item $F(u)\star F(v)\curlyeqprec F(u\ast v)$ for all $u,v\in\mathscr{V}.$
		\end{itemize}
		
	\end{definition}

	\section{Preserving functions between quantales}\label{sec:pf}

	We next introduce a kind of function between quantales that seems to be apropriate for extending the notion of aggregation function.

	\begin{definition}\label{def:catpreserving}
		A map $F:(\mathscr{V},\preceq,\ast)\to(\mathscr{W},\curlyeqprec,\star)$ between two quantales is said to be \textbf{
			$\mathsf{Cat}$-preserving} if the map $\mathsf{F}:\mathscr{V}\text{-}\mathsf{Cat}\to\mathscr{W}\text{-}\mathsf{Cat},$ which assign to a $\mathscr{V}$-category $(X,a)$ the pair $(X,F\circ a),$ and to a $\mathscr{V}$-functor $f:(X,a)\to (Y,b)$ the map $\mathsf{F}(f):(X,F\circ a)\to (Y,F\circ b)$ given by $\mathsf{F}(f)(x)=f(x)$ for all $x\in X,$ is a functor.
		
		We will denote by $\mathscr{C}((\mathscr{V},\preceq,\ast),(\mathscr{W},\curlyeqprec,\star))$, or simply by $\mathscr{C}(\mathscr{V},\mathscr{W})$ if no confusion arises, the family of $\mathsf{Cat}$-preserving functions between the quantales $(\mathscr{V},\preceq,\ast)$ and $(\mathscr{W},\curlyeqprec,\star).$
		
		Under the above conditions, if we consider that the map $\mathsf{F}$ is restricted to the separated $\mathscr{V}$-categories and the separated $\mathscr{W}$-categories we say that $F$ is \textbf{separately $\mathsf{Cat}$-preserving}.
		
		In a similar way, we establish the concept of \textbf{symmetrically $\mathsf{Cat}$-preser\-ving function} when only symmetric categories are considered.
	\end{definition}
	
	\begin{definition}\label{def:preserving}
		A map $F:(\mathscr{V},\preceq,\ast)\to(\mathscr{W},\curlyeqprec,\star)$ between two quantales is said to be \textbf{preserving} if the map $\mathsf{F}:\mathrm{Obj}(\mathscr{V}\text{-}\mathsf{Cat})\to\mathrm{Obj}(\mathscr{W}\text{-}\mathsf{Cat}),$ which assign to a $\mathscr{V}$-category $(X,a)$ the pair $(X,F\circ a)$ is well-defined, that is, if $(X,F\circ a)$ is a $\mathscr{W}$-category.
		
		We will denote by $\mathscr{P}((\mathscr{V},\preceq,\ast),(\mathscr{W},\curlyeqprec,\star))$, or simply by $\mathscr{P}(\mathscr{V},\mathscr{W})$ if no confusion arises, the family of preserving functions between the quantales $(\mathscr{V},\preceq,\ast)$ and $(\mathscr{W},\curlyeqprec,\star).$
		
		If the map $F$ satisfies that $(X,F\circ a)$ is a separated (resp. symmetric) $\mathscr{W}$-category whenever $(X,a)$ is a separated (resp. symmetric) $\mathscr{V}$-category then $F$ is said to be \textbf{separately preserving} (resp. \textbf{symmetrically preserving}).  The family of separately (resp. symmetrically) preserving functions will be denoted by $\mathsf{se}\mathscr{P}((\mathscr{V},\preceq,\ast),(\mathscr{W},\curlyeqprec,\star))$ (resp. $\mathsf{sy}\mathscr{P}((\mathscr{V},\preceq,\ast),(\mathscr{W},\curlyeqprec,\star))$).
	\end{definition}
	
	We next present some examples of the family of $\mathsf{Cat}$-preserving functions $\mathscr{C}(\mathscr{V},\mathscr{W})$ and preserving functions $\mathscr{P}(\mathscr{V},\mathscr{W})$ for particular quantales $\mathscr{V}$ and $\mathscr{W}.$
	
	\begin{example}
		Let $F:(\mathbf{2},\leq,\ast)\to (\mathbf{2},\leq,\ast)$ be a function. Then $F$ is $\mathsf{Cat}$-preserving if and only if $F$ is the identity or $F$ is the identically 1 function. Otherwise, $F(1)=0$ and in this case, given a $\mathbf{2}$-category $(X,a)$ (a partially ordered set, Example \ref{ex:vcat})  and $x\in X$, then $(X,F\circ a)$ is not a $\mathbf{2}$-category since $1\not\leq F(a(x,x))=F(1)=0$ failing to fulfill (VC1).
	\end{example}

	\begin{example}\label{ex:preservingm}
		The concept of preserving function $F:\mathsf{P}_+\to \mathsf{P}_+$ is equivalent to that of an extended quasi-pseudometric aggregation function \cite{MassaVale13}, that is, a function satisfying that if $(X,d)$ is an extended quasi-pseudometric space (a $\mathsf{P}_+$-category) then $(X,F\circ d)$ is also an extended quasi-pseudometric space.
		
		If we consider separately preserving functions, symmetrically preserving functions, or separately and symmetrically preserving functions we obtain extended quasi-metric aggregation functions, extended pseudometric aggregation functions, and extended metric aggregation functions respectively.

		Moreover, consider the family $\mathsf{QPMA}$ of quasi-pseudometric aggregation functions, that is, functions  $G:[0,+\infty)\to [0,+\infty)$ verifying that  $(X,G\circ d)$ is a quasi-pseudometric space whenever $(X,d)$ so is. Then this family embeds into $\mathscr{P}(\mathsf{P}_+,\mathsf{P}_+)$. In fact, for each $G\in \mathsf{QPMA}$ consider the function $\overline{G}:\mathsf{P}_+\to \mathsf{P}_+$ given by 
		$$\overline{G}(x)=\begin{cases} G(x) &\text{ if }x\in [0,+\infty),\\ +\infty &\text{ if }x=+\infty.\end{cases}$$ 
		Then $\overline{G}\in \mathscr{P}(\mathsf{P}_+,\mathsf{P}_+).$ Therefore, 
		
		$$\begin{array}{ccc}
			
			\iota:\mathsf{QPMA}&\to&\mathscr{P}(\mathsf{P}_+,\mathsf{P}_+)\\[5pt]
			G&\to&\overline{G}
		\end{array}$$
		is well-defined and $\iota(\mathsf{QPMA})$ is the subfamily of functions $F$ of $\mathscr{P}(\mathsf{P}_+,\mathsf{P}_+)$ verifying 
		$F^{-1}(+\infty)=\{+\infty\}.$ 
		
		
	\end{example}
	
	The above example can be generalized when the quantale of the domain is $\mathsf{P}_+^I$. To describe this family we give the following definition.
	
	\begin{definition}[\mbox{see \cite[Definition 1]{MassaVale13}}]\label{def:eqpmaf}
		A function $F: [0,+\infty]^I\rightarrow [0,+\infty]$ is said to be 
		\begin{itemize}
			\item an \textbf{extended quasi-pseudometric aggregation function on products} if whenever $\big\{(X_i,d_i):i\in I\big\}$ is a family of extended quasi-pseudometric spaces, then $F\circ d_\Pi$ is an extended quasi-pseudometric on $\prod_{i\in I} X_i$ where $d_\Pi:\left(\prod_{i\in I} X_i\right)\times \left(\prod_{i\in I} X_i\right)\to [0,+\infty]^I$ is given by
			$$d_\Pi(x,y)=(d_i(x_i,y_i))_{i\in I}$$
			for every $x,y\in \prod_{i\in I} X_i.$
			\item an \textbf{extended quasi-pseudometric aggregation function on sets} if whenever set $\{d_{i}:i \in I\}$ is a family of extended quasi-pseudometrics on a nonempty set $X$, then $F \circ d_\Delta$ is an extended quasi-pseudometric on $X$, where $d_\Delta: X \times X \rightarrow [0,+\infty]^{I}$ is given by $$d_\Delta(x,y)=(d_{i}(x,y))_{i \in I}$$
			for every $x,y \in X.$
		\end{itemize}
		
		Similar definitions can be stated for extended quasi-metrics, extended pseudometrics and extended metrics.
	\end{definition}
	
	\begin{remark}\label{rem:qpmaf}
		Extended quasi-metric aggregation functions were studied by Massanet and Valero in \cite{MassaVale13}.
		
		Moreover, if in the previous definition we consider  functions  $F:[0,+\infty)^I\rightarrow [0,+\infty)$ (only defined for finite nonnegative real values) and (quasi-)(pseudo)\-metrics, we obtain the definition of a (quasi-)(pseudo)metric aggregation function on products or on sets. These families of functions have been previously studied in the literature. For example,
		\begin{itemize}
			\item metric aggregation functions on products were studied mainly by Dobo\v{s} \cite{Dobos98} under the name metric preserving functions (see also \cite{Corazza99}); 
			\item quasi-metric aggregation functions on products were studied by Mayor and Valero  \cite{MayorVale10} under the name asymmetric distance aggregation function;
			\item metric aggregation functions on sets were studied by Mayor and Valero  \cite{MayorVale19} under the name $n$-metric preserving functions; 
			\item quasi-metric aggregation functions on sets were studied by Mi\~nana and Valero  \cite{MinyaVale19} under the name $n$-quasi-metric aggregation function.
		\end{itemize}
		
	\end{remark}
	
	We next show that, as in Example \ref{ex:preservingm}, the family $\mathscr{P}(\mathsf{P}_+^I,\mathsf{P}_+)$ is equal to the family of extended quasi-pseudometric aggregation functions on products or on sets (we will see later, in Theorem \ref{thm:plmq}, that $\mathscr{P}(\mathsf{P}_+^I,\mathsf{P}_+)=\mathscr{C}(\mathsf{P}_+^I,\mathsf{P}_+)$).

	\begin{lemma}\label{lem:preserving_eqpmaf}
		The following statements are equivalent:
		\begin{enumerate}
			\item $F:\mathsf{P}_+^I\to \mathsf{P}_+$ is preserving;
			\item $F:[0,+\infty]^I\to [0,+\infty]$ is an extended quasi-pseudometric aggregation function on products;
			\item $F:[0,+\infty]^I\to [0,+\infty]$ is an extended quasi-pseudometric aggregation function on sets.
		\end{enumerate}
		Consequently,
		$$\mathscr{P}(\mathsf{P}_+^I,\mathsf{P}_+)=\mathsf{EQPMAP}^I=\mathsf{EQPMAS}^I,$$
		where $\mathsf{EQPMAP}^I$ {\rm(}resp. $\mathsf{EQPMAS}^I${\rm)} denotes the family of extended quasi-pseudometric aggregation functions on products {\rm(}resp. on sets{\rm)} defined on $[0,+\infty]^I$.
	\end{lemma}
	
	\begin{proof}
		(1) $\Rightarrow$ (2) Let $\big\{(X_i,d_i):i\in I\big\}$ be an arbitrary family of extended quasi-pseudometric spaces, that is,  a family of $\mathsf{P}_+$-categories.  Then $(\prod_{i\in I} X_i, d_\Pi)$ is a $\mathsf{P}_+^I$-category (see Example \ref{ex:prodqpd}.(1)). Since $F$ is preserving then $(\prod_{i\in I} X_i,F\circ d_\Pi)$ is a $\mathsf{P}_+$-category, that is, an extended quasi-pseudometric space. Therefore, $F$ is an extended quasi-pseudometric aggregation on products.
		
		\noindent (2) $\Rightarrow$ (3) This is straightforward.
		
		\noindent (3) $\Rightarrow$ (1) Suppose that $F:[0,+\infty]^I\to [0,+\infty]$ is an extended quasi-pseudometric aggregation function on sets. Given a $\mathsf{P}_+^I$-category $(X,a)$ then $\big\{(X,a_i):i\in I\big\}$ is a family of extended quasi-pseudometric spaces where $a_i$ is the $i$th coordinate function of $a$ (see Example \ref{ex:prodqpd}.(3)). By assumption  $(X,F\circ a_\Pi)=(X,F\circ a)$ is an extended quasi-pseudometric space, that is, a  $\mathsf{P}_+$-category. Therefore, $F$ is preserving.
	\end{proof}

	\begin{remark}\label{rem:extended}
		Since for any sets of indices $I$, the pair $([0,+\infty)^I,\leq^{op})$ is not a complete lattice, then $([0,+\infty)^I,\leq^{op},+)$ it is not a quantale. Thus quasi-pseudometric aggregation functions $F:[0,+\infty)^I\to [0,+\infty)$ (see Remark \ref{rem:qpmaf}) do not adapt to the framework of preserving functions between quantales. 
		
		However, as in Example \ref{ex:preservingm}, the family $\mathsf{QPMAP}^I$ of quasi-pseudometric aggregation functions on products embeds into $\mathscr{P}(\mathsf{P}_+^I,\mathsf{P}_+)$. For any $F\in \mathsf{QPMAP}^I$
		consider the function $\overline{F}:\mathsf{P}_+^I\to \mathsf{P}_+$ given by
		$$\overline{F}(x)=\begin{cases}
			F(x)&\text{ if }x_i\neq +\infty\text{ for all }i\in I,\\
			+\infty&\text{ otherwise },
		\end{cases}$$
		for every $x\in [0,+\infty]^I.$ An easy computation shows that $\overline{F}$ is an extended quasi-pseudometric aggregation function on products, so by the previous result, $\overline{F}:\mathsf{P}_+^I\to \mathsf{P}_+$ is preserving.
		
		Since the map
		
		$$\begin{array}{ccc}
			
			\iota:\mathsf{QPMAP}^I&\to&\mathscr{P}(\mathsf{P}_+^I,\mathsf{P}_+)\\[5pt]
			F&\to&\overline{F}
		\end{array}$$
		is also injective then it is an embedding and $\iota(\mathsf{QPMAP}^I)$ is the subfamily of functions $F$ of $\mathscr{P}(\mathsf{P}_+^I,\mathsf{P}_+)$ verifying 
		$F^{-1}(+\infty)=\{x\in [0,+\infty]^I:x_j=+\infty\text{ for some }j\in I\}.$ 
		
		Notice that, since $\mathscr{P}(\mathsf{P}_+^I,\mathsf{P}_+)$ is also the family of extended quasi-pseudometric aggregation functions on sets, the previous procedure also works for the family $\mathsf{QPMAS}^I$ of quasi-pseudometric aggregation functions on sets and shows that $\mathsf{QPMAP}^I=\mathsf{QPMAS}^I.$
	\end{remark}
	
	\begin{remark}\label{rem:spreserving_epmaf}
		Lemma \ref{lem:preserving_eqpmaf} remains true if we replace \textquote{preserving} by \textquote{symmetrically preserving} and \textquote{extended quasi-pseudometric aggregation function} by \textquote{extended pseudometric aggregation function}. 
		
		We also observe that Pradera and Trillas \cite{PraTri02} previously noticed that pseudometric aggregation function on products and pseudometric aggregation function on sets are equivalent concepts. 
	\end{remark}
	
	At this point, it is natural to wonder whether  Lemma \ref{lem:preserving_eqpmaf} is still valid when we consider separately preserving functions. In this case, at this moment, we can only obtain the following result (first inclusion's proof is similar to (1) $\Rightarrow$ (2) of Lemma \ref{lem:preserving_eqpmaf}).

	\begin{lemma}\label{lem:preserving_eqmaf}
		If $F:\mathsf{P}_+^I\to \mathsf{P}_+$ is separately preserving then it is an extended quasi-metric aggregation function on products.

		Consequently,
		$$\mathsf{se}\mathscr{P}(\mathsf{P}_+^I,\mathsf{P}_+)\subseteq\mathsf{EQMAP}^I\subseteq\mathsf{EQMAS}^I,$$
		where $\mathsf{EQMAP}^I$ {\rm(}resp. $\mathsf{EQMAS}^I${\rm)} denotes the family of extended quasi-metric aggregation functions on products {\rm(}resp. on sets{\rm)} defined on $[0,+\infty]^I.$
	\end{lemma}

	Later, on Theorem \ref{thm:eqmaf_spf} we will show that $\mathsf{se}\mathscr{P}(\mathsf{P}_+^I,\mathsf{P}_+)=\mathsf{EQMAP}^I.$ Nonetheless, we can now provide an example of an extended quasi-metric aggregation function on sets which does not belong to $\mathsf{se}\mathscr{P}(\mathsf{P}_+^I,\mathsf{P}_+)$.
	
	\begin{example}\label{ex:sets_no_product}
		Let $F:[0,+\infty]^2\to [0,+\infty]$ given by
		$$F(x_1,x_2)=\begin{cases}
			0&\text{ if }x_1=0,\\
			1&\text{ if }x_1\neq 0,
		\end{cases}$$
		for every $(x_1,x_2)\in [0,+\infty]^2.$
		
		Let $X$ be a nonempty set and $d_1,d_2$ be two extended quasi-metrics on $X$. Let us check that $F\circ d_\Delta$ is an extended quasi-metric on $X.$ 
		It is obvious that $F\circ d_\Delta(x,x)=F(d_1(x,x),d_2(x,x))=F(0,0)=0$ for all $x\in X.$ Moreover, suppose that 
		$$F(d_\Delta(x,y))=F(d_\Delta(y,x))=F(d_1(x,y),d_2(x,y))=F(d_1(y,x),d_2(y,x))=0$$ for some $x,y\in X.$ By definition of $F,$ $d_1(x,y)=d_1(y,x)=0$ so $x=y$ since $d_1$ is an extended quasi-metric on $X.$ 
		
		Finally, given $x,y,z\in X,$ if $d_1(x,y)=0$ then it is clear that
		$$F\circ d_\Delta(x,y)=0\leq F\circ d_\Delta(x,z)+F\circ d_\Delta(z,y).$$
		If $d_1(x,y)\neq 0$ then, by the triangle inequality of $d_1$, we infer that $d_1(x,z)\neq 0$ or $d_1(z,y)\neq 0$. Hence $F\circ d_\Delta(x,z)\neq 0$ or $F\circ d_\Delta(z,y)\neq 0$ which proves the triangle inequality in this case.
		
		Therefore, $F\circ d_\Delta$ is an extended quasi-metric on $X,$ so $F$ is an extended quasi-metric aggregation function on sets.
		
		However, let us consider the separated $\mathsf{P}_+^2$-category $(Y,a)$ where $Y=\{y_1,y_2\}$ and $a:Y\times Y\to \mathsf{P}_+^2$ is given by
		$$a(y_i,y_j)=\begin{cases}
			(0,0)&\text{ if  }i= j \text{ or }i=1,j=2,\\
			(0,1)&\text{ if }i=2,j=1.
		\end{cases}$$
		Then $(Y,F\circ a)$ is not a separated $\mathsf{P}_+$-category since $y_1\neq y_2$ but
		$$F\circ a(y_1,y_2)=F(0,0)=0=F(0,1)=F\circ a(y_2,y_1).$$
		Consequently, $F$ is not a separately preserving function between the quantales $\mathsf{P}_+^2$ and $\mathsf{P}_+.$
		
	\end{example}

	For obtaining a characterization of the preserving functions between quantales, we will make use of the following concepts which are generalizations of previous notions given in \cite{Dobos98,MayorVale10,PRLVale21}.

		\begin{definition}
		Let $(S,\leq,\cdot)$ be an ordered semigroup, that is, a semigroup endowed with a partial order compatible with the operation. A triplet $(x,y,z)\in S^3$ is said to be an \textbf{asymmetric triangle triplet} on $(S,\leq,\cdot)$ if $x\cdot y\leq z.$ Moreover, it is said to be a \textbf{triangle triplet} if every permutation of the triplet is an asymmetric triangle triplet.
		
		A function $F:(S,\leq,\cdot)\to(M,\leqslant,\ast)$ between two ordered semigroups is said to \textbf{preserve (asymmetric) triangle triplets} if $(F(x),F(y),F(z))$ is a(n) (asymmetric) triangle triplet on $(M,\leqslant,\ast)$ whenever $(x,y,z)$  is a(n) (asymmetrict) triangle triplet on $(S,\leq,\cdot).$
	\end{definition}

%
	
	Next
	we prove that preserving functions between commutative integral quantales are exactly lax morphisms of quantales. 
	
	\begin{theorem}\label{thm:plmq}
		Let $(\mathscr{V},\preceq,\ast),(\mathscr{W},\curlyeqprec,\star)$ be two commutative integral quantales. The following statements are equivalent:
		\begin{enumerate}[(1)]
			\item $F:(\mathscr{V},\preceq,\ast)\to(\mathscr{W},\curlyeqprec,\star)$ is $\mathsf{Cat}$-preserving; 
			\item  $F:(\mathscr{V},\preceq,\ast)\to(\mathscr{W},\curlyeqprec,\star)$ is preserving; 
			\item $F:(\mathscr{V},\preceq,\ast)\to(\mathscr{W},\curlyeqprec,\star)$ is a lax morphism;
			\item $F$ preserves asymmetric triangle triplets and $1_\mathscr{W}=F(1_\mathscr{V}).$
		\end{enumerate}
		Therefore,
		$$\mathscr{P}(\mathscr{V},\mathscr{W})=\mathscr{C}(\mathscr{V},\mathscr{W}).$$
	\end{theorem}
	
	\begin{proof}
		(1) $\Rightarrow$ (2) This is obvious.
		
		(2) $\Rightarrow$ (3)	Suppose that $F$ is preserving. Let $(X,b)$ be an arbitrary $\mathscr{V}$-category. By assumption $(X,F\circ b)$ is a $\mathscr{W}$-category so $1_\mathscr{W}=(F\circ b)(x,x)=F(1_\mathscr{V})$ for any $x\in X.$
		
		Let $u,v\in\mathscr{V}.$ Consider a set $X=\{x_1,x_2,x_3\}$ with three different elements and define $a:X\times X\to\mathscr{V}$ as
		$$a(x_i,x_j)=\begin{cases}
			1_\mathscr{V}&\text{ if }i=j,\\
			u&\text{ if }i=1, j=2,\\
			v&\text{ if }i=2, j=3 ,\\
			u\ast v&\text{ otherwise },
		\end{cases}$$
		for every $i,j\in\{1,2,3\}.$
		Then it is easy to check that  $(X,a)$ is a $\mathscr{V}$-category (in fact, it is a separated $\mathscr{V}$-category except when $u=v=1_\mathscr{V}$). By hypothesis, $(X,F\circ a)$ is a $\mathscr{W}$-category so
		\begin{align*}
			(F\circ a(x_1,x_2))\star (F\circ a(x_2,x_3))&\curlyeqprec F\circ a(x_1,x_3),\\
			F(u)\star F(v)&\curlyeqprec F (u\ast v).
		\end{align*}
		
		On the other hand, let $u,v\in \mathscr{V}$ such that $u\preceq v$ and $u\neq v.$ Again, let $X=\{x_1,x_2,x_3\}$ be a set with three different elements and define $c:X\times X\to\mathscr{V}$ as
		$$c(x_i,x_j)=\begin{cases}
			1_\mathscr{V}&\text{ if }i=j,\\
			u&\text{ if }i=1, j=2,\\
			1_\mathscr{V}&\text{ if }i=2, j=3 ,\\
			v&\text{ if }i=1, j=3 ,\\
			u\ast u&\text{ otherwise },
		\end{cases}$$
		for every $i,j\in\{1,2,3\}.$ 	
		It is straightforward to verify that $(X,c)$ is a $\mathscr{V}$-category (in fact, it is always a separated $\mathscr{V}$-category since $u\neq 1_\mathscr{V}$).  Since $(X,F\circ c)$ is a $\mathscr{W}$-category we deduce that
		\begin{align*}
			(F\circ c(x_1,x_2))\star (F\circ c(x_2,x_3))&\curlyeqprec F\circ c(x_1,x_3),\\
			F(u)\star F(1_\mathscr{V})=F(u)\star 1_\mathscr{W}=F(u)&\curlyeqprec F (v).
		\end{align*}
		
		Consequently, $F$ is isotone and it is a lax morphism of quantales.\medskip
		
		\noindent	(3) $\Rightarrow$ (4)
		Suppose that $(u,v,w)$ is an asymmetric triangle triplet in $(\mathscr{V},\preceq,\ast)$, that is, $u\ast v\preceq w.$ Since $F$ is a lax morphism of quantales then
		$$F(u)\star F(v)\curlyeqprec F(u\ast v)\curlyeqprec F(w)$$
		so $(F(u),F(v),F(w))$ is an asymmetric triangle triplet in $(\mathscr{W},\curlyeqprec,\star).$
		
		Furthermore, since $F$ is a lax morphism then $1_\mathscr{W}=F(1_{\mathscr{V}}).$

		\medskip

		\noindent	(4) $\Rightarrow$ (1)
		Let $(X,a)$ be a $\mathscr{V}$-category. Given $x,y,z\in X$ then $a(x,z)\ast a(z,y)\preceq a(x,y)$. Hence $(a(x,z),a(z,y),a(x,y))$ is an asymmetric triangle triplet in $(\mathscr{V},\preceq,\ast)$. By hypothesis $(F(a(x,z)),F(a(z,y)),F(a(x,y)))$ is an asymmetric triangle triplet on $(\mathscr{W},\curlyeqprec,\star)$, that is,
		$$(F\circ a(x,z))\star (F\circ a(z,y))\curlyeqprec F\circ a(x,z).$$
		Moreover, $1_\mathscr{W}= F(1_\mathscr{V})=F\circ a(x,x)$ for all $x\in X.$ Therefore, $(X,F\circ a)$ is a $\mathscr{W}$-category.
		
		On the other hand, if $f:(X,a)\to (Y,b)$ is a $\mathscr{V}$-functor then, for every $x,y\in X$ it holds $a(x,y)\preceq b(f(x),f(y))$ so $(a(x,y),1_\mathscr{V},b(f(x),b(f(y)))$ is an asymmetric triangle triplet in $(\mathscr{V},\preceq,\ast)$. Using that $F$ preserves asymmetric triangle triplets we deduce
		$$F(a(x,y))=F(a(x,y))\star 1_\mathscr{W}= F(a(x,y))\star F(1_\mathscr{V})\curlyeqprec F(b(f(x),f(y)))$$ so $\mathsf{F}(f)$ is a $\mathscr{W}$-functor. Consequently, $F$ is $\mathsf{Cat}$-preserving.
	\end{proof}
	
	\begin{remark}
		We have previously noted that, using the terminology from \cite{Lawv73}, a lax morphism of quantales is a specific type of closed functor between closed categories. Lawvere pointed out \cite[p. 149]{Lawv73} that these closed functors induce functors between the corresponding enriched categories, thereby demonstrating the implication (3) $\Rightarrow$ (2) (see also, for example, \cite{HofReis13}).
	\end{remark}

	If we particularize the previous result for the quantales $\mathsf{P}_+^I$ and $\mathsf{P}_+,$ we obtain the following result that characterizes extended quasi-pseudometric aggregation functions. We must emphasize that this result seems to have been forgotten in the literature about the aggregation of metric structures.
	
	\begin{theorem}\label{thm:eqpmaf}
		Let $F:\mathsf{P}_+^I\to\mathsf{P}_+.$ The following statements are equivalent:
		\begin{enumerate}[(1)]
			\item $F$ is an extended quasi-pseudometric aggregation function on products;
			\item $F$ is an extended quasi-pseudometric aggregation function on sets;
			\item $F$ is a lax morphism;
			\item $F$ preserves asymmetric triangle triplets and $F((0)_{i\in I})=0.$
		\end{enumerate}
	\end{theorem}

	\begin{proof}
		It is a direct consequence of Lemma \ref{lem:preserving_eqpmaf} and Theorem \ref{thm:plmq}.    
	\end{proof}

	It is reasonable to wonder whether we can use our results for obtaining a characterization of functions aggregating quasi-pseudometrics instead of extended quasi-metrics. If $F:[0,+\infty)^I\to [0,+\infty)$ is a quasi-pseudometric aggregation function on products (or on sets) let us consider the function $\overline{F}:\mathsf{P}_+^I\to \mathsf{P}_+$ given by
	$$\overline{F}(x)=\begin{cases}
		F(x)&\text{ if }x_i\neq +\infty\text{ for all }i\in I,\\
		+\infty&\text{ otherwise },
	\end{cases}$$
	for every $x\in [0,+\infty]^I$ (see Remark \ref{rem:extended}). It is straightforward to check that:
	\begin{itemize}
		\item $\overline{F}$  is an extended quasi-pseudometric aggregation function on products (or on sets) if and only if $F$ is a quasi-pseudometric aggregation function on products (or on sets);
		\item $\overline{F}$  preserves asymmetric triangle triplets if and only if $F$ preserves asymmetric triangle triplets;
		\item $\overline{F}$ is a lax morphism if and only if $F$ is isotone, subadditive (that is, $F(x+y)\leq F(x)+F(y)$ for all $x,y\in [0,+\infty)^I$) and $F((0)_{i\in I})=0.$
	\end{itemize}
	
	In this way, we can obtain the following result. 
	
	\begin{theorem}\label{thm:qpmaf}
		Let $F:[0,+\infty)^I\to [0,+\infty).$ The following statements are equivalent:
		\begin{enumerate}[(1)]
			\item $F$ is a quasi-pseudometric aggregation function on products;
			\item $F$ is a quasi-pseudometric aggregation function on sets;
			\item $F$ is isotone, subadditive and $F((0)_{i\in I})=0;$
			\item $F$ preserves asymmetric triangle triplets and $F((0)_{i\in I})=0.$
		\end{enumerate}
	\end{theorem}

	For separately preserving functions, we can prove the following result.
	
	\begin{theorem}\label{thm:seplmq}
		
		Let $(\mathscr{V},\preceq,\ast),(\mathscr{W},\curlyeqprec,\star)$ be two commutative integral quantales. The following statements are equivalent:
		\begin{enumerate}[(1)]
			\item $F:(\mathscr{V},\preceq,\ast)\to(\mathscr{W},\curlyeqprec,\star)$ is separately $\mathsf{Cat}$-preserving; 
			\item  $F:(\mathscr{V},\preceq,\ast)\to(\mathscr{W},\curlyeqprec,\star)$ is separately preserving; 
			\item $F:(\mathscr{V},\preceq,\ast)\to(\mathscr{W},\curlyeqprec,\star)$ is a lax morphism satisfying $F^{-1}\left(1_{\mathscr{W}}\right)=\{1_{\mathscr{V}}\}.$
			\item $F$ preserves asymmetric triangle triplets and $F^{-1}\left(1_{\mathscr{W}}\right)=\{1_{\mathscr{V}}\}.$
		\end{enumerate}
		Therefore,
		$$\mathsf{se}\mathscr{P}(\mathscr{V},\mathscr{W})=\mathsf{se}\mathscr{C}(\mathscr{V},\mathscr{W}).$$
	\end{theorem}
	
	\begin{proof}
		The proof of this result mimics that of Theorem \ref{thm:plmq} but we have to justify the addition of  $F^{-1}\left(1_{\mathscr{W}}\right)=\{1_{\mathscr{V}}\}$ in statements (3) and (4). Suppose that (2) holds. Assume, by way of contradiction, that there exists $v\in\mathscr{V}$ different from $1_\mathscr{V}$ such that $F(v)=1_\mathscr{W}.$ Consider $X=\{x,y\}$ a set with two different elements and define $a:X\times X\to\mathscr{V}$ as $a(x,y)=a(y,x)=v$ and $a(x,x)=a(y,y)=1_\mathscr{V}.$ Then $(X,a)$ is a separated $\mathscr{V}$-category. Nevertheless, $(X,F\circ a)$ is not separated since $1_\mathscr{W}=F(v)=F\circ a(x,y)=F\circ a(y,x)$ but $x\neq y.$ 
		
		(4) $\Rightarrow$ (1)  Given a separated $\mathscr{V}$-category $(X,a)$, following Theorem \ref{thm:plmq}'s proof, you obtain that $(X,F\circ a)$ is a $\mathscr{W}$-category. Let us check that it is separated. Let $x,y\in X$ such that $1_\mathscr{W}= F\circ a(x,y)=F\circ a(y,x).$ By assumption,  $a(x,y)=a(y,x)=1_\mathscr{V}$ so $x=y$ since $(X,a)$ is separated. 
	\end{proof}	
	
	Specifying the above result for the quantales $\mathsf{P}_+^I$ and $\mathsf{P}_+,$ we obtain the following result that characterizes extended quasi-metric aggregation functions. 
	
	\begin{theorem}\label{thm:eqmaf_spf}
		Let $F:\mathsf{P}_+^I\to\mathsf{P}_+.$ The following statements are equivalent:
		\begin{enumerate}[(1)]
			\item $F$ is an extended quasi-metric aggregation function on products;
			\item $F$ is a lax morphism and $F^{-1}(0)=\{(0)_{i\in I}\};$
			\item $F$ preserves asymmetric triangle triplets and $F^{-1}(0)=\{(0)_{i\in I}\};$
			\item $F$ is a separately preserving function.
		\end{enumerate}
		Therefore,
		$$\mathsf{se}\mathscr{P}(\mathsf{P}_+^I,\mathsf{P})=\mathsf{EQMAP}^I.$$
	\end{theorem}

	\begin{proof}
		(1) $\Rightarrow$ (2) By \cite[Theorem 9]{MassaVale13} we only need to show that $F$ is isotone. Let $a,b\in [0,+\infty]^I$ with $a\leq b.$ Let $J=\{i\in I: b_i\neq 0\}.$ For each $i\in I$ consider the set
		$$X_i=\begin{cases}
			\{t_1,t_2,t_3\}&\text{ if }i\in J,\\
			\{t_1\}&\text{ if }i\not\in J,
		\end{cases}$$
		where $t_1,t_2,t_3$ are different elements. For each $i\in I,$ define $d_i:X_i\times X_i\to [0,+\infty]$ as:
		\begin{itemize}
			\item if $i\in J$ then
			$$d_i(t_j,t_k)=\begin{cases}
				0&\text{ if }j=k,\\
				b_i&\text{ if }j=1,k=2,\\
				0&\text{ if }j=2,k=3,\\
				a_i&\text{ if }j=1,k=3,\\
				b_i+b_i&\text{ otherwise };\\
			\end{cases}$$
			\item if $i\not\in J$ then
			$$d_i(t_1,t_1)=0.$$
		\end{itemize}
		An easy computation shows that $d_i$ is an extended quasi-metric on $X_i$ for all $i\in I.$ By assumption $F\circ d_\Pi$ is an extended quasi-metric on $\prod_{i\in I}X_i.$
		
		Now, let us consider $x,y,z\in \prod_{i\in I} X_i$ given by 
		$$x_i=t_1,\hspace*{1cm}y_i=\begin{cases} t_3&\text{ if }i\in J,\\
			t_1&\text{ if }i\not\in J,\end{cases}\hspace*{1cm}z_i=\begin{cases} t_2&\text{ if }i\in J,\\
			t_1&\text{ if }i\not\in J,\end{cases}
		$$
		
		for all $i\in I.$ Therefore, by the triangle inequality,
		\begin{align*}
			(F\circ d_\Pi)(x,y)&\leq (F\circ d_\Pi)(x,z)+(F\circ d_\Pi)(z,y),\\
			F((d_i(x_i,y_i))_{i\in I})&\leq F((d_i(x_i,z_i))_{i\in I})+F((d_i(z_i,y_i))_{i\in I}),\\
			F(a)&\leq F(b)+F((0)_{i\in I})=F(b).
		\end{align*}
		Hence, $F$ is isotone.

		(2) $\Rightarrow$ (3) follows from Theorem \ref{thm:eqpmaf}. 
		
		(3) $\Rightarrow$ (4) $F$ is a preserving function by Theorem \ref{thm:eqpmaf}. Since $F^{-1}(0)=\{(0)_{i\in I}\},$ an easy verification shows that $F$ is also separately preserving.
		
		(4) $\Rightarrow$ (1) It is a consequence of Lemma \ref{lem:preserving_eqmaf}.
	\end{proof}

	We emphasize that the above result contradicts \cite[Example 3]{MassaVale13}, where the authors provide an example of an extended quasi-metric aggregation function on products that it is not isotone. Nevertheless, the example was not correct as we next show.
	
	\begin{example}[\mbox{\cite[Example 3]{MassaVale13}}]
		Let $F:[0,+\infty]^2\to [0,+\infty]$ given by
		$$F(x)=\begin{cases}
			+\infty&\text{ if }x_1>0,\\
			2&\text{ if }x_1=0,x_2\in\;]0,1[,\\
			1&\text{ if }x_1=0,x_2\geq 1,\\
			0&\text{ if }x_1=x_2=0,
		\end{cases}$$
		for all $x=(x_1,x_2)\in [0,+\infty]^2.$ In \cite[Example 3]{MassaVale13}, the authors assert that $F$ is an extended quasi-metric aggregation function. Nonetheless, this is not true. In fact, let $X=\{x_1,x_2,x_3\}$ be a set with three different elements and define $d_1,d_2:X\times X\to [0,+\infty)$ as
		$$d_1(x_i,x_j)=\begin{cases}
			0&\text{ if }i\leq j,\\
			1&\text{ if }i>j,
		\end{cases}\hspace*{1.5cm}
		d_2(x_i,x_j)=\begin{cases}
			0&\text{ if }i=j,\\
			1&\text{ if }i=1,j=2,\\
			0&\text{ if }i=2,j=3,\\
			\frac{1}{2}&\text{ if }i=1,j=3,\\
			2&\text{otherwise.}
		\end{cases}$$
		It is straightforward to check that $d_1,d_2$ are quasi-metrics on $X.$ But $F\circ (d_1,d_2)$ is not a quasi-metric on $X^2$ since it does not verify the triangle inequality. For example
		\begin{align*}
			(F\circ (d_1,d_2))&((x_1,x_1),(x_3,x_3))=F(d_1(x_1,x_3),d_2(x_1,x_3))=F(0,\tfrac{1}{2})=2\\
			&\not\leq (F\circ (d_1,d_2))((x_1,x_1),(x_2,x_2))+(F\circ (d_1,d_2))((x_2,x_2),(x_3,x_3))\\
			&=F(d_1(x_1,x_2),d_2(x_1,x_2))+F(d_1(x_2,x_3),d_2(x_2,x_3))\\
			&=F(0,1)+F(0,0)=1+0=0.
		\end{align*}
		Therefore, $F$ is not an extended quasi-metric aggregation function on products.
	\end{example}

	As above, we would like to obtain a characterization of functions aggregating quasi-metrics on products from the previous results. It was shown in \cite[Example 4]{MassaVale13} that there exist extended quasi-metric aggregation functions on products that are not quasi-metric aggregation functions on products. Nonetheless, as before, if $F:[0,+\infty)^I\to [0,+\infty)$ is a quasi-metric aggregation function on products then $\overline{F}$  is an extended quasi-metric aggregation function on products if and only if $F$ is a quasi-metric aggregation function on products.

	Consequently, we can obtain the following result first proved in \cite[Theorems 6 and 11]{MayorVale10}.
	
	\begin{theorem}[\cite{MayorVale10}]
		Let $F:[0,+\infty)^I\to [0,+\infty).$ The following statements are equivalent:
		\begin{enumerate}[(1)]
			\item $F$ is a quasi-metric aggregation function on products;
			\item $F$ is isotone, subadditive and $F^{-1}(0)=\{(0)_{i\in I}\};$
			\item $F$ preserves asymmetric triangle triplets and $F^{-1}(0)=\{(0)_{i\in I}\}.$           
		\end{enumerate}
	\end{theorem}
	
	We have previously observed (see Example \ref{ex:sets_no_product}) that extended quasi-metric aggregation functions on sets are different from extended quasi-metric aggregation functions on products. In fact, the same example allows to check that quasi-metric aggregation functions on sets are different from quasi-metric aggregation functions on products. A characterization of quasi-metric aggregation function on sets was obtained in \cite[Theorem 2.6]{MinyaVale19}. A similar characterization can be obtained in the extended case.\medskip

	Now, we center our attention on the symmetrically preserving functions.
	
	\begin{theorem}\label{thm:symp}
		Let $(\mathscr{V},\preceq,\ast),(\mathscr{W},\curlyeqprec,\star)$ be two commutative integral quantales. Then a function $F:(\mathscr{V},\preceq,\ast)\to(\mathscr{W},\curlyeqprec,\star)$ is symmetrically preserving if and only if $F$ preserves triangle triplets and $1_\mathscr{W}=F(1_\mathscr{V}).$
	\end{theorem}
	
	\begin{proof}
		Suppose that $F$ is symmetrically preserving. That $1_\mathscr{W}= F(1_\mathscr{V})$ follows as in the proof of (2) $\Rightarrow$ (3) of Theorem \ref{thm:plmq}.
		
		Let $(u,v,w)$ be a triangle triplet in  $(\mathscr{V},\preceq,\ast)$.  Consider a set $X=\{x_1,x_2,x_3\}$ with three different elements and define $a:X\times X\to\mathscr{V}$ as
		$$a(x_i,x_j)=a(x_j,x_i)=\begin{cases}
			1_\mathscr{V}&\text{ if }i=j,\\
			u&\text{ if }i=1, j=2,\\
			v&\text{ if }i=2, j=3 ,\\
			w&\text{ if }i=1, j=3,
		\end{cases}$$
		for every $i,j\in\{1,2,3\}.$
		An straightforward computation shows that $(X,a)$ is a symmetric $\mathscr{V}$-category. By hypothesis $(X,F\circ a)$ is a symmetric $\mathscr{W}$-category. Consequently,
		\begin{align*}
			F(a(x_1,x_2))\star F(a(x_2,x_3))&\curlyeqprec F(a(x_1,x_3)),\\
			F(u)\star F(v)&\curlyeqprec F(w).
		\end{align*}
		By permuting the elements $x_1,x_2,x_3$ we obtain that $(F(u),F(v),F(w))$ is a triangle triplet on $(\mathscr{W},\curlyeqprec,\star)$. So $F$ preserves triangle triplets.
		
		The converse is similar to the implication (4) $\Rightarrow$ (1) of Theorem \ref{thm:plmq}.
	\end{proof}
	
	By Remark \ref{rem:spreserving_epmaf} we obtain the following consequence of the previous theorem.
	
	\begin{corollary}\label{cor:epmaf}
		$F:[0,+\infty]^I\to[0,+\infty]$ is an extended pseudometric aggregation function on products (or on sets) if and only if $F$ preserves triangle triplets and $F((0)_{i\in I})=0.$ 
	\end{corollary}
	
	\begin{remark}
		We notice that the above result is very similar to the characterization obtained by Dobo\v{s} \cite{Dobos98} about the metric aggregation functions on products. The only difference relies in that in this case, $F$ vanishes only at $(0)_{i\in I}$. This was previously observed in \cite{PraTri02}, when the authors studied the pseudometric aggregation functions on products and on sets.
		
		Moreover, given a function $F:[0,+\infty)^I\to [0,+\infty)$ we have already commented that it has, in some sense, a similar behaviour to $\overline{F}.$ This has allowed us to deduce a characterization for functions aggregating quasi (pseudo)-metrics from the extended cases. This also occurs if $F$ is a pseudometric aggregation function on products or on sets, so Corollary \ref{cor:epmaf} also is true for pseudometric aggregation functions on products or on sets.
	\end{remark}
	
	For obtaining a characterization of the symmetrically $\mathsf{Cat}$-preserving functions, we have to add an extra ingredient to the properties satisfied by the symmetrically preserving functions.
	
	\begin{theorem}\label{thm:csymp}
		Let $(\mathscr{V},\preceq,\ast),(\mathscr{W},\curlyeqprec,\star)$ be two commutative unital quantales. Then a function $F:(\mathscr{V},\preceq,\ast)\to(\mathscr{W},\curlyeqprec,\star)$ is symmetrically $\mathsf{Cat}$-preserving if and only if  $F$ is isotone, preserves triangle triplets and $1_\mathscr{W}= F(1_\mathscr{V}).$
	\end{theorem}
	
	\begin{proof}
		Suppose that $F:(\mathscr{V},\preceq,\ast)\to(\mathscr{W},\curlyeqprec,\star)$ is symmetrically $\mathsf{Cat}$-preserving and let $u,v\in\mathscr{V}$ such that $u\preceq v.$ Let $X=\{x_1,x_2\}$ be a set with two different elements and define $a,b:X\times X\to \mathscr{V}$ as
		$$a(x_i,x_j)=\begin{cases}
			1_\mathscr{V}&\text{ if }i=j,\\
			u&\text{ if }i\neq j,\\
		\end{cases}\hspace*{2cm} b(x_i,x_j)=\begin{cases}
			1_\mathscr{V}&\text{ if }i=j,\\
			v&\text{ if }i\neq j,\\
		\end{cases}$$ 
		for all $i,j\in\{1,2\}.$ Then $(X,a), (X,b)$ are $\mathscr{V}$-categories. Let $f:(X,a)\to (Y,b)$ be the identity function. It is clear that $f$ is a $\mathscr{V}$-functor since $a(x_1,x_2)=u\preceq b(f(x_1),f(x_2))=b(x_1,x_2).$
		
		By assumption $\mathsf{F}(f): (X,F\circ a)\to (X,F\circ b)$ is a $\mathscr{W}$-functor so
		\begin{align*}
			(F\circ a)(x_1,x_2)&\preceq (F\circ b)(\mathsf{F}(f)(x_1),\mathsf{F}(f)(x_2))=(F\circ b)(x_1,x_2),\\
			F(u)&\preceq F(v).
		\end{align*}
		Hence $F$ is isotone.
		
		That $F$ preserves triangle triplets and $F(1_\mathscr{V})=1_\mathscr{W}$ is similar the proof of Theorem \ref{thm:symp}.
		
		Conversely, again by Theorem \ref{thm:symp}, we only have to prove that whenever $f:(X,a)\to (Y,b)$ is a $\mathscr{V}$-functor between two $\mathscr{V}$-categories then $f:(X,F\circ a)\to (Y,F\circ b)$ is a $\mathscr{W}$-functor. But this is obvious since $a(x,y)\preceq b(f(x),f(y))$ implies $F(a(x,y))\curlyeqprec F(b(f(x),f(y)))$ for all $x,y\in X$ since $F$ is isotone.
	\end{proof}
	
	From the previous results we can infer the following.
	
	\begin{theorem}\label{thm:psCp}
		$F:(\mathscr{V},\preceq,\ast)\to(\mathscr{W},\curlyeqprec,\star)$ is preserving if and only if it is symmetrically $\mathsf{Cat}$-preserving.
	\end{theorem}
	
	\begin{proof}
		Suppose that $F$ is preserving. Since it is obvious that a function preserving asymmetric triangle triplets also preserves triangle triplets, we deduce from Theorems \ref{thm:plmq} and \ref{thm:symp} that $F$ is symmetrically preserving. Moreover, $F$ is also isotone since given $u,v\in \mathscr{V}$ with $u\preceq v$ then $(u,1_\mathscr{V},v)$ is an asymmetric triangle triplet so, by Theorem \ref{thm:plmq}, $(F(u),F(1_\mathscr{V}),F(v))=(F(u),1_\mathscr{W},F(v))$ is also an asymmetric triangle triplet. Hence $F(u)\curlyeqprec F(v).$ The conclusion follows from Theorem \ref{thm:csymp}.
		
		Conversely, suppose that $F$ is $\mathsf{Cat}$-preserving. By Theorem \ref{thm:csymp} we have that $F(1_\mathscr{V})=1_\mathscr{W},$ $F$ is isotone and preserves triangle triplets. Then if $(u,v,w)$ is an asymmetric triangle triplet then $u\ast v\preceq w$ so $F(u\ast v)\curlyeqprec F(w).$ Moreover, $(u,v,u\ast v)$ is a triangle triplet so $(F(u),F(v),F(u\ast v))$ also is. Hence $F(u)\star F(v)\curlyeqprec F(u\ast v).$ Consequently, $F(u)\star F(v)\curlyeqprec F(w),$ that is, $(F(u),F(v),F(w))$ is an asymmetric triangle triplet. The result follows from Theorem \ref{thm:plmq}.
	\end{proof}
	
	Therefore, we can fuse Theorems \ref{thm:plmq} and \ref{thm:psCp} to obtain the next result.
	
	\begin{corollary}\label{cor:preserving_spreserving}
		Let $(\mathscr{V},\preceq,\ast),(\mathscr{W},\curlyeqprec,\star)$ be two commutative integral quantales. The following statements are equivalent:
		\begin{enumerate}[(1)]
			\item $F:(\mathscr{V},\preceq,\ast)\to(\mathscr{W},\curlyeqprec,\star)$ is $\mathsf{Cat}$-preserving; 
			\item  $F:(\mathscr{V},\preceq,\ast)\to(\mathscr{W},\curlyeqprec,\star)$ is preserving; 
			\item $F:(\mathscr{V},\preceq,\ast)\to(\mathscr{W},\curlyeqprec,\star)$ is symmetrically $\mathsf{Cat}$-preserving; 
			\item $F:(\mathscr{V},\preceq,\ast)\to(\mathscr{W},\curlyeqprec,\star)$ is a lax morphism;
			\item $F$ preserves asymmetric triangle triplets and $1_\mathscr{W}=F(1_\mathscr{V}).$
			\item $F$ preserves symmetric triangle triplets, it is isotone and $1_\mathscr{W}=F(1_\mathscr{V}).$
		\end{enumerate}
		Therefore,
		$$\mathscr{P}(\mathscr{V},\mathscr{W})=\mathscr{C}(\mathscr{V},\mathscr{W})=\mathsf{sy}\mathscr{C}(\mathscr{V},\mathscr{W}).$$
	\end{corollary}
	
	\begin{proof}
		It is a direct consequence of Theorems \ref{thm:plmq}, \ref{thm:csymp} and \ref{thm:psCp}.
	\end{proof}
	
	Notice that any of the above equivalent statements implies that $F:(\mathscr{V},\preceq,\ast)\to(\mathscr{W},\curlyeqprec,\star)$ is symmetrically preserving. However, we cannot add this assertion to the corollary as the next example shows (another example a\-ppears in \cite[Proposition 2.9]{Corazza99}).
	
	\begin{example}
		Let $\mathbf{3}=\{\bot,x,\top\}$ be a set with three different elements endowed with the partial order $\preceq$ determined by $\bot\preceq x\preceq \top.$ Furthermore, define the binary operation $\ast$ on $\mathbf{3}$ given by 
		$$u\ast v=\begin{cases}
			u\wedge v&\text{ if }u=\top\text{ or }v=\top,\\
			\bot&\text{ otherwise},
		\end{cases}$$
		for every $u,v\in\mathbf{3}.$ Then $(\mathbf{3},\preceq,\ast)$ is a commutative integral quantale. 
		
		Consider the function $F:(\mathbf{3},\preceq,\ast)\to (\mathbf{3},\preceq,\ast)$
		given by $F(\top)=\top$, $F(x)=\bot$ and $F(\bot)=x.$
		
		Suppose that $(u,v,w)$ is a triangle triplet. If all the elements of the triplet are different from $\top$ then $(F(u),F(v),F(w))$ is $(\bot,\bot,\bot)$, $(\bot,\bot,x)$, $(\bot,x,x)$, $(x,x,x)$  or a permutation of them, and all of them are triangle triplets. 
		
		If, for example, $u=\top$ then $u\ast v=v\preceq w$ and $u\ast w=w\preceq v$ so $v=w.$ Then it is obvious that  $(F(u),F(v),F(w))=(\top,F(v),F(v))$ is a triangle triplet. Hence $F$ preserves triangle triplets. By Theorem \ref{thm:symp}, $F$ is symmetrically preserving.
		
		Nevertheless, $F$ does not preserve asymmetric triangle triplets. Notice that $(\bot,\top,x)$ is an asymmetric triangle triplet but $(F(\bot),F(\top),F(x))=(x,\top,\bot)$ is not.
		
	\end{example}

	\begin{remark}\label{rem:diff}
		The above result sheds light on the already existent results about the difference between quasi-pseudometric aggregation functions on products (or on sets) (see Theorem \ref{thm:eqpmaf} and \ref{thm:qpmaf}) and pseudometric aggregation functions on products (see Corollary \ref{cor:epmaf}). It is clear that every quasi-pseudometric aggregation function on products is also a pseudometric aggregation function on products, although the converse is not true in general. The difference between these two family of functions is that, in some sense, quasi-pseudometric aggregation function on products preserve also some types of morphisms that are not preserved by pseudometric aggregation functions on products.
	\end{remark}

	To finish the paper, we study the particular case of the family $\mathscr{P}(\Delta^I_+(\ast),$ $\Delta_+(\ast))$ of preserving functions between the quantales $\Delta^I_+(\ast)$ and $\Delta_+(\ast)$, where $\ast$ is a t-norm. We saw in Example \ref{ex:vcat}.(3) that $\Delta_+(\ast)$-categories are fuzzy quasi-pseudometrics spaces with respect to the t-norm $\ast.$ The next example also shows that functions belonging to the family $\mathscr{P}(\Delta^I_+(\ast),\Delta_+(\ast))$ allow to aggregate fuzzy quasi-pseudometrics.
	
	\begin{remark}\label{rem:preserving-fqpmafp}
		Given a continuous t-norm $\ast$, let $F\in\mathscr{P}(\Delta_+^I(\ast), \Delta_+(\ast))$ be a preserving function. Let $\big\{(X_i,M_i,\ast):i\in I\big\}$ be a family of fuzzy quasi-pseudometric spaces. Then $\big\{(X_i,m_i):i\in I\big\}$ is a family of $\Delta_+(\ast)$-categories (see Example \ref{ex:vcat}.(3)). Therefore, by Example \ref{ex:prodqpd} and Remark \ref{rem:prodfqpd}, $(\prod_{i\in I} X_i,m_\Pi)$ is a $\mathsf{\Delta}_+^I(\ast)$-category where
		$$m_\Pi(x,y)=(m_i(x_i,y_i))_{i\in I}$$ for all $x,y\in \prod_{i\in I}X_i.$
		By assumption $(\prod_{i\in I} X_i, F\circ m_\Pi) $ is a $\Delta_+(\ast)$-category. Hence, by Example \ref{ex:vcat}.(3), $(M_{F\circ m_\Pi},\ast)$ is a fuzzy quasi-pseudometric on $\prod_{i\in I}X_i$, where
		\begin{align*}
			M_{F\circ m_\Pi}(x,y,t)&=(F\circ m_\Pi)(x,y)(t)=F((m_i(x,y))_{i\in I})(t)
		\end{align*}
		for every $x,y\in\prod_{i\in I}X_i$, $t\geq 0.$ 
	\end{remark}
	
	The above example naturally leads to the following definition.
	
	\begin{definition}\label{def:fqpmaf_quantale}
		A function $F:\Delta_+^I(\ast)\to \Delta_+(\ast)$ is said to be:
		\begin{itemize}
			\item a \textbf{$\ast$-fuzzy quasi-pseudometric aggregation function on products} if  whenever $\Big\{(X_i,M_i,\ast):i\in I\Big\}$ is a family of fuzzy quasi-pseudometric spaces then $(M_{F\circ m_\Pi},\ast)$ is a fuzzy quasi-pseudometric on $\prod_{i\in I}X_i$ where
			$$M_{F\circ m_\Pi}(x,y,t)=F((m_i(x_i,y_i))_{i\in I})(t)$$
			for all $x,y\in \prod_{i\in I}X_i$, $t\geq 0;$
			\item a \textbf{$\ast$-fuzzy quasi-pseudometric aggregation function on sets} if whenever  $\Big\{(M_i,\ast):i\in I\Big\}$ is a family of fuzzy quasi-pseudometrics on the same set $X$ then $(M_{F\circ m_\Delta},\ast)$ is a fuzzy quasi-pseudometric on $X$ where  
			$$M_{F\circ m_\Delta}(x,y,t)=F((m_i(x,y))_{i\in I})(t)$$
			for all $x,y\in X$, $t\geq 0.$
		\end{itemize}
	\end{definition}
	
	We next show that the family $\mathscr{P}(\Delta_+^I(\ast), \Delta_+(\ast))$ is exactly the family of $\ast$-fuzzy quasi-pseudometric aggregation function on products or on sets.

	\begin{theorem}\label{thm:fqpmaf}
		Let $\ast$ be a t-norm and let $F:\Delta_+^I(\ast)\to \Delta_+(\ast)$ be a map. The following statements are equivalent:
		\begin{enumerate}[(1)]
			\item $F$ is a fuzzy quasi-pseudometric aggregation function on products;
			\item $F$ is a fuzzy quasi-pseudometric aggregation function on sets;
			\item $F$ is preserving;
			\item $F$ is $\mathsf{Cat}$-preserving;
			\item $F$ is a lax morphism of quantales.
		\end{enumerate}
	\end{theorem}
	
	\begin{proof}
		(1) $\Rightarrow$ (2) This is straightforward.
		
		(2) $\Rightarrow$ (3) Let $(X,a)$ be a $\Delta_+^I(\ast)$-category. For each $i\in I,$ let $a_i$ be the $i$th-coordinate function of $a$. Then $\Big\{(X,a_i):i\in I\Big\}$ is a family of $\Delta_+(\ast)$-categories. Consequently, $\Big\{(M_{a_i},\ast):i\in I\Big\}$ is a family of fuzzy quasi-pseudometrics on $X$ (see Example \ref{ex:vcat}). By assumption, $(M_{F\circ m_\Delta},\ast)$ is a fuzzy quasi-pseudometric on $X$ where  
		$$M_{F\circ m_\Delta}(x,y,t)=F((m_{a_i}(x,y))_{i\in I})(t)=F((a_i(x,y))_{i\in I})(t)$$
		for all $x,y\in X$, $t\geq 0.$ By Example \ref{ex:vcat}.(3), $(X,m_{F\circ m_\Delta})$ is a $\Delta_+(\ast)$-category and 
		$$m_{F\circ m_\Delta}(x,y)(t)=F(a(x,y))(t)=(F\circ a)(x,y)(t)$$
		for all $x,y\in X$, $t\in [0,+\infty].$ Notice that in the case $t=+\infty$ we have that
		\begin{align*}
			m_{F\circ m_\Delta}(x,y)(+\infty)&=\bigvee_{0\leq s} M_{F\circ m_\Pi}(x,y,s)=\bigvee_{0\leq s} F(a(x,y))(s)\\&=F(a(x,y))(+\infty)
		\end{align*}
		where in the last equality we have use that $F(a(x,y))\in \Delta_+(\ast).$ 
		
		Therefore, $F$ is preserving. 
		
		The equivalence between (3), (4) and (5) is a consequence of Theorem \ref{thm:plmq}. 
		
		(3) $\Rightarrow$ (1) has been shown in Remark \ref{rem:preserving-fqpmafp}.
	\end{proof}

	However, in \cite{PRLVale21}, we can find another method for aggregating fuzzy quasi-pseudometrics that is analogous to that of aggregating quasi-pseudometrics (see Definition \ref{def:eqpmaf} and Remark \ref{rem:qpmaf}). The difference relies on the fact that in \cite{PRLVale21}, the aggregation functions are from $[0,1]^I$ to $[0,1]$, meanwhile in Definition \ref{def:fqpmaf_quantale} we have considered functions between the quantales $\Delta_+^I(\ast)$ and $\Delta_+(\ast).$ We recall the definition given in \cite{PRLVale21}.

	\begin{definition}[\cite{PRLVale21}] \label{def:fqpmaf}
		A function $F:[0,1]^{I}\rightarrow [0,1]$ is said to be:
		\begin{itemize}
			\item a \textbf{fuzzy (quasi-)(pseudo)metric aggregation function on pro\-ducts} if whenever $\ast$ is a t-norm and $\Big\lbrace (X_{i}, M_{i}, \ast): i \in I\Big\rbrace$ is a family of fuzzy (quasi-)(pseudo)metric spaces then $(F \circ M_\Pi, \ast)$ is a fuzzy (quasi-)(pseudo)metric on $\Pi_{i \in I}X_{i}$, where $M_\Pi: (\Pi_{i \in I}X_{i})^{2} \times [0, +\infty) \rightarrow [0,1]^{I}$ is given by
			$$M_\Pi(x,y,t)=(M_{i}(x_{i}, y_{i}, t))_{i\in I}$$
			for every $x, y \in \Pi_{i \in I}X_{i}$ and $t \geq 0$.
			If $F$ only satisfies the above condition for a fixed t-norm $\ast$, then it is said to be an $\ast$-fuzzy \mbox{(quasi-)}(pseudo)metric aggregation function on products.
			\item a \textbf{fuzzy (quasi-)(pseudo)metric aggregation function on sets} if whenever $\ast$ is a t-norm and $\Big\lbrace (M_{i}, \ast) : i \in I \Big\rbrace $ is a family of fuzzy (quasi-)(pseudo)metrics on the same set $X$ then $(F \circ M_\Delta, \ast)$ is a fuzzy (quasi-)(pseudo)metric on $X$, where $M_\Delta: X^{2} \times [0,+\infty) \rightarrow [0,1]^{I}$ is given by 
			$$M_\Delta(x,y,t)=(M_{i}(x,y,t))_{i\in I}$$
			for every $x,y \in X$ and $t \geq 0$. If $F$ only satisfies the above condition for a fixed t-norm $\ast$, then it is said to be an $\ast$-fuzzy \mbox{(quasi-)}(pseudo)metric aggregation function on sets.
		\end{itemize}
	\end{definition}

	At this point, it is natural to wonder about the relationship between these two methods for aggregating fuzzy quasi-pseudometrics.
	We will see in Proposition \ref{prop::DeltaEmbedding} that the family of fuzzy quasi-pseudometric aggregation functions on sets from $[0,1]^I$ to $[0,1]$ can be considered as a subfamily of the preserving functions $\mathscr{P}(\Delta_+^I(\ast),\Delta_+(\ast)).$ We will use the following result.

	\begin{proposition}\label{prop:Fdelta}
		Let $\ast$ be a continuous t-norm and $F:([0,1]^I,\leq,\ast)\to([0,1],\leq,\ast)$ be a map. Define  $F_\Delta:\Delta_+^I(\ast)\to \Delta_+(\ast)$ 
		as 
		$$F_\Delta((f_i)_i)(t)=F((f_i(t))_i)$$ for every $(f_i)_i\in \Delta_+^I(\ast)$ and every $t\in [0,+\infty].$
		
		Then $F_\Delta$ is a lax morphism if and only if $F$ is a left-continuous lax morphism.
	\end{proposition}
	
	\begin{proof}
		Suppose that $F$ is a left-continuous lax morphism. We first check that $F_\Delta$ is well-defined, that is, $F_\Delta((f_i)_i)\in\Delta_+(\ast)$ for any $(f_i)_i\in \Delta_+^I(\ast).$ Given $t,s\in [0,+\infty]$ such that $t\leq s$ then $(f_i(t))_i\leq (f_i(s))_i$ since $f_i\in \Delta_+(\ast)$ for every $i\in I$. Hence, $F_\Delta((f_i)_i)(t)=F((f_i(t))_i)\leq F((f_i(s))_i)=F_\Delta((f_i)_i)(s)$ since $F$ is isotone. Therefore $F_\Delta((f_i)_i)$ is isotone. 
		
		Furthermore, for every $t\in[0,+\infty]$ and $i\in I,$ since $f_i\in\Delta_+(\ast)$ we have that $f_i(t)=\bigvee_{0\leq s<t} f_i(s).$ Since $F$ is left-continuous then 
		\begin{align*}
			F_\Delta((f_i)_i)(t)&=F((f_i(t))_i)=F\left(\left(\bigvee_{0\leq s<t} f_i(s)\right)_i\right)=\bigvee_{0\leq s<t} F((f_i(s))_i)\\&=\bigvee_{0\leq s<t} F_\Delta((f_i)_i)(s).\end{align*}
		Consequently, $F_\Delta((f_i)_i)\in\Delta_+(\ast).$

		Next, we check that $F_\Delta$ is a lax morphism of quantales. 	Let $(g_i)_i,(h_i)_i\in \Delta_+^I(\ast)$ such that $(g_i)_i\leq (h_i)_i.$ Since $F$ is isotone then $F_\Delta((g_i)_i)(t)=F((g_i(t))_i)\leq F((h_i(t))_i)=F_\Delta((h_i)_i)(t)$ for all 
		$t\in [0,+\infty]$ . Thus $F_\Delta$ is isotone.
		
		Let $g=(g_i)_i,h=(h_i)_i\in \Delta_+^I(\ast).$ Let $t\in [0,+\infty]$ and $r,s\in [0,+\infty]$ such that $r+s\leq t.$ Since $F$ is a lax morphism then $F(g(r))\ast F(h(s))\leq F(g(r)\ast h(s))$. Consequently
		\begin{align*}			
			(F_\Delta(g)\circledast F_\Delta(h))(t)&=\bigvee_{r+s\leq t} F_\Delta(g)(r)\ast F_\Delta(h)(s)=\bigvee_{r+s\leq t} F((g_i(r))_i)\ast F((h_i(s))_i)\\&\leq \bigvee_{r+s\leq t} F((g(r)_i\ast h_i(s))_i)
			\leq F\left(\left(\bigvee_{r+s\leq t} g_i(r)\ast h_i(s)\right)_i\right)\\&=F(((g_i\circledast h_i)(t))_i)\\
			&=F_\Delta\left(g\circledast h\right)(t),
		\end{align*}
		
		and since $t$ is arbitrary
		$$ F_\Delta(g)\circledast F_\Delta(h)\leq F_\Delta(g\circledast h).$$
		
		Finally, 
		$$F_\Delta((f_{0,1})_i)(t)=F((f_{0,1}(t))_i)=\begin{cases}
			F((1)_i)&\text{ if }t>0,\\
			F((0)_i)&\text{ if }t=0,
		\end{cases}=\begin{cases}
			1&\text{ if }t>0,\\
			0&\text{ if }t=0,
		\end{cases}=f_{0,1},$$
		where we have used that $F((1)_i)=1$ since $F$ is a lax morphism and $F((0)_i)=F(\bigvee \varnothing)=\bigvee F(\varnothing)=0$ since $F$ is left-continuous.
		
		Therefore $F_\Delta$ is a lax morphism.

		Conversely, suppose that $F_\Delta$ is a lax morphism. Given $\alpha\in [0,1]$ define $f_\alpha\in\Delta_+(\ast)$ as
		$$f_\alpha(t)=\begin{cases}\alpha&\text{ if }t>0,\\
			0&\text{ if }t=0.\end{cases}$$
		Let $(u_i)_i,(v_i)_i\in [0,1]^I$ such that $(u_i)_i\leq (v_i)_i$. Then $(f_{u_i})_i\leq (f_{v_i})_i.$ Since $F_\Delta$ is isotone then $F_\Delta((f_{u_i})_i)\leq F_\Delta((f_{v_i})_i)$ so \begin{align*}
			F_\Delta((f_{u_i})_i)(1)&\leq F_\Delta((f_{v_i})_i)(1),\\
			F((f_{u_i}(1))_i)&\leq	F((f_{v_i}(1))_i),\\
			F((u_i)_i)&\leq F((v_i)_i).
		\end{align*}
		Hence $F$ is isotone.
		
		Moreover, if $(u_i)_i,(v_i)_i\in [0,1]^I$ are arbitrary then 
		\begin{align*}
			(F_\Delta((f_{u_i})_i)\circledast F_\Delta((f_{v_i})_i))(2)&\leq F_\Delta((f_{u_i})_i\circledast(f_{v_i})_i )(2),\\
			\bigvee_{r+s\leq 2} F_\Delta((f_{u_i})_i)(r)\ast  F_\Delta((f_{v_i})_i)(s)&\leq F\left(\bigvee_{r'+s'\leq 2} (f_{u_i}(r'))_i\ast (f_{v_i}(s'))_i\right).
		\end{align*}
		Since $(f_{u_i}(t))_i=(u_i)_i$ and $(f_{v_i}(t))_i=(v_i)_i$ for every $t>0$, the above supremum is attained for any pair $0<r+s\leq 2$ so 
		\begin{align*}
			F_\Delta((f_{u_i})_i)(1)\ast  F_\Delta((f_{v_i})_i)(1)&\leq F((f_{u_i}(1))_i\ast (f_{v_i}(1))_i),\\
			F((u_i)_i)\ast F((v_i)_i)&\leq F((u_i)_i\ast (v_i)_i).
		\end{align*}
		Furthermore, since $F_\Delta$ is a lax morphism then $F_\Delta((f_{0,1})_i)=f_{0,1}.$ Hence $F((1)_i)=F((f_{0,1}(1))_i)=F_\Delta((f_{0,1})_i)(1)=f_{0,1}(1)=1$. Consequently, $F$ is a lax morphism.
		
		At last, let us check that $F$ is left-continuous. Let $(t_i)_i\in [0,1]^I.$ If $t_i=0$ for all $i\in I$, since $F_\Delta(f_0)\in\Delta_+(\ast)$ then $0=F_\Delta((f_0)_i)(0)=F((f_0(0))_i)=F((0)_i).$ Suppose now that there exists $j\in I$ with $t_j\neq 0.$ For each $i\in I$ define $h_i\in \Delta_+(\ast)$ as
		$$h_i(t)=\begin{cases}
			t_i&\text{ if }t>1,\\
			t_i\cdot t&\text{ if }0\leq t\leq 1.
		\end{cases}$$
		Since $F_\Delta((h_i)_i)\in\Delta_+(\ast)$ then $F_\Delta((h_i)_i)(t)=\bigvee_{0\leq s<t}F_\Delta((h_i)_i)(s)$ for every $t\in [0,+\infty].$ In particular,
		\begin{align*}
			F_\Delta((h_i)_i)(1)&=\bigvee_{0\leq s<1}F_\Delta((h_i)_i)(r)\\
			F((h_i(1))_i)&=\bigvee_{0\leq r<1}F((h_i(r))_i)\\
			F((t_i)_i)&=\bigvee_{0\leq r<1}F((t_i\cdot r)_i)
		\end{align*}
		Since $\{(t_i\cdot r)_i:r\in [0,1)\}\subseteq \{(s_i)_i\in [0,1]^I: (0)_i\leq (s_i)_i<(t_i)\}$ and $F$ is isotone then 
		$$F((t_i)_i)=\bigvee_{0\leq r<1}F((t_i\cdot r)_i)\leq\bigvee_{(s_i)_i<(t_i)_i}F((s_i)_i)\leq F((t_i)_i) $$ so $F$ is left-continuous.
	\end{proof}

	\begin{proposition}\label{prop::DeltaEmbedding}
		Let $\ast$ be a continuous t-norm and let $\mathsf{FQPMAFS}^I(\ast)$ be the family of functions $F:[0,1]^I\to [0,1]$ that are $\ast$-fuzzy quasi-pseudometric aggregation functions on sets. Then the map
		$$\begin{array}{ccc}
			
			\iota:\mathsf{FQPMAFS}^I(\ast)&\to&\mathscr{P}(\Delta_+^I(\ast),\Delta_+(\ast))\\[5pt]
			F&\to&F_\Delta
		\end{array}$$
		is an embedding.    
	\end{proposition}

	\begin{proof}
		If $F\in \mathsf{FQPMAFS}(\ast)$, by Theorem \cite[Theorem 4.21]{PRLVale21}, $F:([0,1]^I,\leq,\ast)\to ([0,1],\leq,\ast)$ is a lax morphism and left-continuous. By Proposition \ref{prop:Fdelta} $F_\Delta$ is a lax morphism so a preserving function by Theorem \ref{thm:plmq}. 
	\end{proof} 
	
	\begin{remark}
		Notice that the above result is also valid for the family of $\ast$-fuzzy (quasi-)pseudometric aggregation function on products since it coincides with $\mathsf{FQPMAFS}^I(\ast)$ (see \cite[Theorem 4.21]{PRLVale21})
	\end{remark}

	However, not every preserving function  $F:\Delta_+^I(\ast)\to \Delta_+(\ast)$ can be expressed as $G_\Delta$ for some $G:[0,1]^I\to [0,1]$ as the next example shows.
	
	\begin{example}
		Let us consider the function $F:(\Delta_+,\leq,\cdot)\to (\Delta_+,\leq,\cdot)$ given by
		$$F(f)=\begin{cases}
			f_0&\text{ if }f(\tfrac{1}{2})>0,\\
			f_{\tfrac{1}{2}}&\text{ if }f(\tfrac{1}{2})=0,
		\end{cases}$$
		where
		$$f_\alpha(t)=\begin{cases}
			1&\text{ if }t>\alpha,\\0&\text{ if }0\leq t\leq \alpha.\end{cases}
		$$
		for every $\alpha\in [0,1]$. It is easy to check that $F$ is preserving. In case we could find $G:[0,1]\to [0,1]$ such that $G_\Delta=F$ then 
		$$G_\Delta(f_{\tfrac{1}{2}})(\tfrac{1}{3})=G(f_{\tfrac{1}{2}}(\tfrac{1}{3}))=G(0)=F(f_{\tfrac{1}{2}})(\tfrac{1}{3})=f_{\tfrac{1}{2}}(\tfrac{1}{3})=0,$$
		$$G_\Delta(f_{\tfrac{1}{3}})(\tfrac{1}{3})=G(f_{\tfrac{1}{3}}(\tfrac{1}{3}))=G(0)=F(f_{\tfrac{1}{3}})(\tfrac{1}{3})=f_{0}(\tfrac{1}{3})=1.$$
		This gives two different values for $G(0)$ which is not possible.
		
	\end{example}


	Notice that $F:[0,1]^I\to [0,1]$ is a fuzzy quasi-pseudometric aggregation function on products (or on sets) if and only if $F_\Delta:\Delta_+^I(\ast)\to \Delta_+(\ast)$ is a fuzzy quasi-pseudometric aggregation function (or on sets). This only requires that Definition \ref{def:fqpmaf_quantale} reduces to Definition \ref{def:fqpmaf} for $F_\Delta.$  Then, using Theorem \ref{thm:fqpmaf} and Proposition \ref{prop:Fdelta} we deduce from our results the following characterization obtained in \cite{PRLVale21}.
	
	\begin{theorem}[\mbox{compare \cite[Theorem 4.21]{PRLVale21}}]
		Let $\ast$ be a t-norm and let $F:([0,1]^I,\leq,\ast)\to([0,1],\leq,\ast)$ be a map. The following statements are equivalent:
		\begin{enumerate}[(1)]
			\item $F$ is a fuzzy quasi-pseudometric aggregation function on products;
			\item $F$ is a fuzzy quasi-pseudometric aggregation function on sets;
			\item $F$ is a left-continuous lax morphism of quantales.
		\end{enumerate}
	\end{theorem}
	
\section{Conclusions and future work}

The analysis of various results in the literature regarding aggregation functions of certain mathematical structures \cite{Corazza99,Dobos98,MayorVale10,MayorVale19,MartinMayorVale11,MassaVale13,MinyaVale19,PRLVale21} reveals some interesting coincidences. This prompted us to explore whether we could establish a general framework for studying these aggregation functions that encompasses some of these findings. The objective of this paper is to develop such a framework.

To achieve this, we have focused on the concept of a category enriched over a quantale \cite{BookMonoidalTopology}, as this concept includes various mathematical structures such as partially ordered sets, extended quasi-pseudometric spaces, and fuzzy quasi-pseudometric spaces, among others. We define preser\-ving functions between quantales as functions that induce a functor between the corresponding enriched categories. We have demonstrated that these functions are equivalent to lax morphisms of quantales, showing that these morphisms can be considered an appropriate extension of aggregation functions.

Furthermore, to illustrate the applicability of this general framework, we have derived results about the aggregation of extended quasi-pseudometrics and fuzzy quasi-pseudometrics based on our framework.

For future research, we aim to explore other mathematical structures where we can apply this theory. Additionally, we intend to investigate potential applications of this theory in areas such as computer science \cite{MayorVale10} and location analysis \cite{MinyaVale19}.

\section*{Acknowledgment}

We sincerely thank the anonymous referees for their valuable comments and suggestions, which have improved the quality of this paper.

	
\end{document}